\pdfoutput=1
\documentclass{article}
\usepackage{amsmath,amsfonts,amssymb,amsthm,mathtools}
\usepackage{microtype}
\usepackage{url,hyperref}
\usepackage{indentfirst}
\hypersetup{colorlinks, linkcolor={red}, citecolor={blue}, urlcolor={black}}
\usepackage{comment}
\usepackage{hyperref}
\usepackage{cleveref}
\usepackage{enumerate,enumitem}
\usepackage{tikz}
\usepackage[margin=1.25in]{geometry}
\usepackage[T1]{fontenc}
\usepackage{adjustbox}
\usepackage{subfig}
\usepackage{graphicx}

\allowdisplaybreaks

\newtheorem{theorem}{Theorem}[section]
\newtheorem{lemma}[theorem]{Lemma}

\newtheorem{corollary}[theorem]{Corollary}

\def\ex{\mathrm{ex}}

\theoremstyle{definition}

\begin{document}

\title{A note on hypergraph extensions of Mantel's theorem}
\date{}

\author{
Jie Ma\thanks{School of Mathematical Sciences, University of Science and Technology of China, Hefei, Anhui 230026, and Yau Mathematical Sciences Center, Tsinghua University, Beijing 100084, China.
Research supported by National Key Research and Development Program of China 2023YFA1010201 and National Natural Science Foundation of China grant 12125106. Email: {\tt jiema@ustc.edu.cn}.}
\and
Tianming Zhu\thanks{School of Mathematical Sciences, University of Science and Technology of China, Hefei, Anhui 230026, China. Research supported by Innovation Program for Quantum Science and Technology 2021ZD0302902.
Email: {\tt zhutianming@mail.ustc.edu.cn}.}
}

\maketitle

\begin{abstract}
Chao and Yu \cite{chao2024entropymeetsturannew} introduced an entropy method for hypergraph Tur\'an problems, and used it to show that the family of $\lfloor k/2\rfloor$ $k$-uniform {\it tents} have Tur\'an density $k!/k^k$. 
Il'kovi\v{c} and Yan \cite{ilkovivc2025improved} improved this by reducing to a subfamily of $\lceil k/e\rceil$ tents. 
In this note, enhancing Il'kovi\v{c}-Yan's result, we give a significantly shorter entropy proof, with optimal bounds within this framework.
\end{abstract}

\section{Introduction}\label{SEC:Intorduction}
\noindent 
The {\it Tur\'{a}n number} $\ex(n,\mathcal{F})$ of a family $\mathcal{F}$ of $k$-uniform hypergraphs ({\it $k$-graphs} for short) denotes the maximum number of edges in an $n$-vertex $k$-graph not containing any member of $\mathcal{F}$ as its subgraph. 
Its {\it Tur\'{a}n density} is given by $\pi (\mathcal{F} ) \coloneqq \lim_{n\to \infty}\text{ex}(n,\mathcal{F})/\binom{n}{k}.$
For integers $n$, let $[n]:=\{1,2,\ldots, n\}.$

The study of Tur\'an number and Tur\'an density of graphs and hypergraphs is one of the central topics in Extremal Combinatorics. 
While the celebrated theorems of Tur\'an \cite{turan1941external} and Erd\H{o}s-Stone-Simonovits \cite{erdos1966limit} completely characterize Tur\'an densities for all graph families,
the hypergraph setting remains largely open, with exact densities known only for very few cases (see \cite{keevash2011hypergraph}). 
Significant progress has been made through sustained investigations of hypergraph extensions of Mantel's theorem concerning $\mathrm{ex}(n,K_3)$. 
These extensions have identified many families of $k$-graphs with Tur\'an density $k!/k^k$. 
This particular value has attracted special attention, owing in part to Erd\H{o}s's famous conjecture regarding whether $k!/k^k$ is a jump for $k$-graphs (see \cite{frankl1984hypergraphs} for details).  
To proceed, for each $i\in [k]$, we define the {\it $(k-i,i)$-tent} $\Delta_{(k-i,i)}$ to be the $k$-graph with vertex set $[2k-1]$ and edge set
\begin{equation}\label{equ:i-tent}
    \big\{ \{1,2,\ldots k\},\{1,\ldots, i,k+1,\ldots, 2k-i  \},\{i+1,\ldots ,k+1,2k-i+1,\ldots, 2k-1\} \big\}.
\end{equation}
More generally, for any partition $\lambda=(\lambda_1,\ldots,\lambda_\ell)$ of $k$ with $\lambda_1\ge\ldots \ge \lambda_\ell\geq 1$, 
the {\it $\lambda$-tent} $\Delta_\lambda$ denotes the $k$-graph with $(k-1)\ell +1$ vertices and $\ell+1$ edges $e,e_1,\ldots e_\ell$ defined as follows:
\begin{itemize}
    \item There exists a vertex $v$ (called {\it apex}) such that $e_i \cap e_j = \{v\}$ for all $1 \leq i < j \leq \ell$;
    \item The subsets $e \cap e_1, \ldots, e \cap e_\ell$ form a partition of $e$, where $|e\cap e_i|=\lambda_i$ for each $i\in [\ell]$.
\end{itemize}
Frankl and F\"uredi \cite{frankl1983new} were the first to determine the exact Tur\'an number for a hypergraph, where they showed $\pi(\Delta_{(2,1)})=3!/3^3$. 
Pikhurko \cite{pikhurko2008exact} proved $\pi(\Delta_{(3,1)})=4!/4^4$ and the exact Tur\'an number for large $n$. 
Generalizing a result of Mubayi \cite{M06}, Mubayi and Pikhurko \cite{mubayi2007new} determined  $\pi(\Delta_{(1,1,\ldots,1)})=k!/k^k$.
Recently, Chao and Yu \cite{chao2024entropymeetsturannew} developed an innovative entropy-based approach to hypergraph Tur\'an problems. 
Applying this method, they \cite{chao2024entropymeetsturannew} proved the following generalization of Mantel's theorem: for every $k \geq 2$, 
the family 
$\mathcal{F}_{k} := \{\Delta_{(k-i,i)} : 1 \leq i \leq \lfloor k/2 \rfloor\}$ satisfies $\pi(\mathcal{F}_{k}) = k!/k^{k}$.
This implies that for any $1\leq q\le \lceil k/2\rceil$, $\pi(\Delta_{(q,1,\ldots,1)})=k!/k^k$, generalizing the above results of \cite{frankl1983new,M06,mubayi2007new}.
Combining the entropy method with other techniques, Liu \cite{liu2025hypergraph} determined the exact Tur\'an number of $\mathcal{F}_{k}$ for large $n$.
For each $1\leq s\leq \lfloor k/2 \rfloor$, define
\begin{equation}
\mathcal{F}_{k}^{\leq s}= \{\Delta_{(k-i,i)} : 1 \leq i\leq  s\},   
\end{equation}
Very recently, Il'kovi\v{c} and Yan \cite{ilkovivc2025improved} improved both of these results in \cite{chao2024entropymeetsturannew,liu2025hypergraph} by determining $\pi(\mathcal{F}_{k}^{\leq \lceil k/e \rceil}) = k!/k^{k}$ for $k\geq 4$ and the exact Tur\'an number of $\mathcal{F}_{k}^{\leq \lceil k/e \rceil}$ for large $n$.

In this note, building on the entropy method of Chao and Yu \cite{chao2024entropymeetsturannew}, we provide a much short proof of the above Tur\'an density results with a slightly better bound.
Our main result is as follows. 
Throughout this note, for any integer $k\geq 2$, define
\begin{equation}\label{equ:t(k)}
\mbox{ $t(k)$ to be the largest integer $s$ such that } 1/s+1/(s+1)+\cdots+1/k>1.
\end{equation}

\begin{theorem}\label{THM:Main}
For every integer $k \geq 2$, $\pi(\mathcal{F}_k^{\le t(k)}) = k!/k^k.$
\end{theorem}

Note that the integer $t=t(k)$ satisfies $1 < \sum_{i=t}^k \frac{1}{i} \le \int_{t-1}^k \frac{1}{x} dx = \ln\frac{k}{t-1}$,
which implies $t(k)\le \lceil k/e \rceil$.
So Theorem~\ref{THM:Main} generalizes the result of Il'kovi\v{c} and Yan \cite{ilkovivc2025improved} that
$\pi(\mathcal{F}_{k}^{\lceil k/e \rceil})=k!/k^k$ for all $k\ge 4$.
Table~\ref{table} compares $\lceil k/e \rceil$ and $t(k)$ for $4 \leq k \leq 19$.

\begin{table}[!ht]
\renewcommand\arraystretch{1.5}
    \begin{minipage}[c]{0.24\textwidth}
    \centering
        \begin{tabular}{l|c|c}
            \hline
        $k$ & $\lceil k/e \rceil$ & $t(k)$ \\ \hline
        4   &   2   &   2   \\
        5   &   2   &   2   \\
        6$^\ast$ &   3   &   2   \\
        7   &   3   &   3   \\
         \hline

    \end{tabular}
    \end{minipage}
    \begin{minipage}[c]{0.24\textwidth}
    \centering
        \begin{tabular}{l|c|c}
            \hline
        $k$ & $\lceil k/e \rceil$ & $t(k)$ \\ \hline
        ~8   &   3   &   3   \\
        ~9$^\ast$ &   4   &   3   \\
        10  &   4   &   4   \\
        11$^\ast$&   5   &   4   \\
        \hline

    \end{tabular}
    \end{minipage}
    \begin{minipage}[c]{0.24\textwidth}
    \centering
        \begin{tabular}{l|c|c}
            \hline
        $k$ & $\lceil k/e \rceil$ & $t(k)$ \\ \hline
        12  &   5   &   5   \\
        13  &   5   &   5   \\
        14$^\ast$&   6   &   5   \\
         15  &   6   &   6   \\
        \hline

    \end{tabular}
    \end{minipage}
    \begin{minipage}[c]{0.24\textwidth}
    \centering
        \begin{tabular}{l|c|c}
            \hline
        $k$ & $\lceil k/e \rceil$ & $t(k)$ \\ \hline
        16  &   6   &   6   \\
        17$^\ast$&   7   &   6   \\
        18  &   7   &   7   \\
        19 &    7 &   7 \\
        \hline

    \end{tabular}
    \end{minipage}
    \caption{Consider the range $4 \le k \le  19$. The asterisk $^\ast$ denotes instances where $t(k) < \lceil k/e \rceil$.}\label{table}
\end{table}

As observed in \cite{chao2024entropymeetsturannew}, results such as Theorem~\ref{THM:Main} can be extended to determine exact Tur\'an densities for certain $\lambda$-tents $\Delta_\lambda$. 
The following corollary provides the complete statement of this extension.
For any partition $\lambda = (\lambda_1,\ldots,\lambda_\ell)$ of $k$, let $\Sigma_\lambda \coloneqq \big\{ \sum_{i\in A} \lambda_i : A \subseteq [\ell] \big\}$ denote its subset sum set.

\begin{corollary}\label{Coro:Main-2}
Let $k \geq 2$, and let $ t = t(k)$ be defined as in~\eqref{equ:t(k)}.
If $\lambda$ is a partition of $k$ satisfying $[t]\subseteq \Sigma_\lambda$, then $\pi (\Delta_\lambda)=k!/k^k$.
\end{corollary}

In particular, this shows that $\pi(\Delta_{(6,2,1)})=9!/9^9$ and for any $q \leq k - t(k)$, $\pi (\Delta_{(q,1,\ldots,1)})=k!/k^k$.

\section{The entropy method of Chao-Yu}
\noindent
In this section, we give a brief overview of the entropy method developed by Chao and Yu~\cite{chao2024entropymeetsturannew}, emphasizing the results most relevant to our presentation.

We begin by introducing some fundamental concepts of entropy.
For any discrete random variable $X$, we denote by $p_X(x)$ the probability $\mathbb{P}(X = x)$.
The {\it Shannon entropy} of $ X $ is defined as
$$\mathbb{H}(X) \coloneqq -\sum_{x \in \mathrm{supp}(X)} p_X(x) \cdot \log_2 p_X(x),$$
where $\mathrm{supp}(X)$ denotes the support of $X$, i.e., the set of all $x$ with $p_X(x) > 0$.
Let $X,Y, X_1, X_2, \ldots, X_n$ be discrete random variables.
We write $\mathbb{H}(X_1, X_2, \ldots, X_n)$ for the entropy of the joint distribution of the random tuple $(X_1, X_2, \ldots, X_n)$.
The {\it conditional entropy} of $X$ given $Y$ is defined as $$\mathbb{H}(X \mid Y) \coloneqq \mathbb{H}(X, Y) - \mathbb{H}(Y).$$

Let $G$ be a $k$-graph.
A $k$-tuple of random vertices $(X_1,\ldots, X_k )\in V(G)^k$ is a {\it random edge with uniform ordering} on $G$,
if $(X_1,\ldots, X_k )$ is symmetric\footnote{That is, the distribution of $(X_{\sigma_1},\ldots, X_{\sigma_k} )$ is always the same for any permutation $\sigma$ of $[k]$.}, and $\{ X_1,\ldots, X_k\}$ is always an edge of $G$.
Using this concept, Chao and Yu \cite{chao2024entropymeetsturannew} established a novel characterization of Tur\'an density.
For two $k$-graphs $F$ and $G$, a {\it homomorphism} from $F$ to $G$ is a mapping $f \colon V(F) \to V(G)$ such that for every edge $e$ in $F$, its image $f(e)$ is an edge in $G$.
Given a family $\mathcal{F}$ of $k$-graphs, we say a $k$-graph $G$ is {\it $\mathcal{F}$-hom-free}, if there exists no homomorphism from any member $F \in \mathcal{F}$ to $G$.

\begin{lemma}[Chao-Yu, {\cite[Corollary 5.6]{chao2024entropymeetsturannew}}]\label{lem:pi}
For any family $\mathcal{F}$ of $k$-graphs, $\pi (\mathcal{F})$ is the supremum of $$2^{\mathbb{H}(X_1,\ldots, X_k )-k\mathbb{H}(X_1)}$$ for any random edge with uniform ordering $(X_1,\ldots, X_k)$ on any $\mathcal{F}$-hom-free $k$-graph $G$.
\end{lemma}

The {\it ratio sequence} $(x_1, \ldots, x_k )$ of a random edge $(X_1, \ldots, X_k )$ with uniform ordering on $G$
is given by $x_i=2^{\mathbb{H}(X_i|X_{i+1},\ldots, X_k )-\mathbb{H}(X_i)}$ for all $i\in [k].$
It is evident to see that $2^{\mathbb{H}(X_1,\ldots,X_k) - k\mathbb{H}(X_1)}=\prod_{i=1}^k x_i$
and $x_k = 1$.
The following result on ratio sequences is implicit in the proof of Lemma~7.2 of \cite{chao2024entropymeetsturannew}.

\begin{lemma}[Chao-Yu, {\cite[Lemma~7.2]{chao2024entropymeetsturannew}}]\label{LEMMA:CY-CONSTRAINT}
Let $1\le t \le \lfloor k/2 \rfloor$.
Let $G$ be any $\mathcal{F}_k^{\le t}$-hom-free $k$-graph, and let $(x_1,\ldots, x_k )$ be the ratio sequence of any random edge with uniform ordering on $G$. Then $x_i+x_j\le x_{i+j}$ holds for all $i,j$ with $i\in [t]$ and $1\leq i+j\le k$.
\end{lemma}

We also need the following well-known property on Tur\'an densities.

\begin{lemma}\label{lem:Hom}
    Let $\mathcal{F},\mathcal{G}$ be two family of $k$-graphs. If for every $G\in \mathcal{G}$, there exists $F \in \mathcal{F}$ such that there exists a homomorphism from $F$ to $G$, then $\pi (\mathcal{F}) \leq \pi(\mathcal{G})$.
\end{lemma}

\section{Proof of the results}
\noindent
The proof of our main theorem relies on the following key lemma.
We note that the statement with $t(k)$ replaced by $\lceil k/e \rceil$ was proved by Il'kovi\v{c} and Yan in \cite[Theorem~3.1]{ilkovivc2025improved}. Here we provide a much shorter proof with a slightly better bound.

\begin{lemma}\label{LEMMA:Final-Inequality}
    Let $k\geq 2$ and $t:=t(k)$ be defined as in~\eqref{equ:t(k)}. Suppose that $y_1,\ldots y_k$ are some nonnegative real numbers with $y_i+y_j\le y_{i+j}$ for all $i\in [t]$ and $1\leq i+j\le k $. Then
    $\prod_{i=1}^k y_i \le \frac{k!}{k^k}y_k^k.$
\end{lemma}

We begin by showing how Lemma~\ref{LEMMA:Final-Inequality} implies Theorem~\ref{THM:Main}.

\begin{proof}[\bf Proof of Theorem~\ref{THM:Main}, using Lemma~\ref{LEMMA:Final-Inequality}.]
    First, observe that any complete $k$-partite $k$-graph is $\mathcal{F}_k^{\le t}$-hom-free, which yields the lower bound $\pi(\mathcal{F}_k^{\le t})\ge k!/k^k$.
    Now, let $G$ be any $\mathcal{F}_k^{\le t}$-free $k$-graph, and let $(x_1,\ldots, x_k )$ be the ratio sequence of any random edge with uniform ordering on $G$.
    Applying Lemmas~\ref{LEMMA:CY-CONSTRAINT} and \ref{LEMMA:Final-Inequality} under the condition $x_k=1$, we deduce that $\prod_{i\in [k]}x_i\leq k!/k^k$.
    By Lemma~\ref{lem:pi}, this immediately gives $\pi(\mathcal{F}_k^{\le t})\leq k!/k^k$, thereby completing the proof.
\end{proof}

Using homomorphism properties of $\lambda$-tents, we derive Corollary~\ref{Coro:Main-2} from Theorem~\ref{THM:Main} as follows.

\begin{proof}[\bf Proof of Corollary~\ref{Coro:Main-2}.]
Let $k\geq 2$ and $t:=t(k)$ be defined as in~\eqref{equ:t(k)}.
Let $\lambda=(\lambda_1,\ldots,\lambda_\ell)$ be a partition of $k$ such that $[t]\subseteq \Sigma_\lambda$.
Recall that $\Delta_\lambda$ has an apex vertex $v$ and edges $e, e_1, \ldots, e_\ell$ satisfying $|e_j \cap e| = \lambda_j$ for each $j \in [\ell]$.
The lower bound $\pi (\Delta_\lambda)\geq k!/k^k$ follows easily by the fact that any complete $k$-partite $k$-graph is $\Delta_\lambda$-hom-free.
For the upper bound, we claim that for any fixed $i\in [t]$, there exists a homomorphism from $\Delta_\lambda$ to the $(k-i,i)$-tent $\Delta_{(k-i,i)}\in \mathcal{F}_k^{\le t}$.
Since we have $[t] \subseteq \Sigma_\lambda$, there is a subset $A\subseteq [k]$ such that $\sum_{j \in A} \lambda_j = i$.
According to the edge set \eqref{equ:i-tent} of $\Delta_{(k-i,i)}$, where $V(\Delta_{(k-i,i)})=[2k-1]$,
we can define a map $f: e \cup \{v\} \to [k+1]$ by setting $f(v) = k+1$ and bijectively mapping $e$ to $[k]$ such that
$f(\bigcup_{j \in A} (e \cap e_j)) = [i]$ and $f(\bigcup_{j \in [\ell]\setminus A} (e \cap e_j)) = [k]\setminus [i]$.
Since each vertex in $V(\Delta_\lambda)\setminus (e \cup \{v\})$ has degree exactly one,
it is easy to extend $f$ to a homomorphism $\tilde{f}: V(\Delta_\lambda) \to V(\Delta_{(k-i,i)})$, which preserves all edges.
This proves the claim.
By Lemma~\ref{lem:Hom} and Theorem~\ref{THM:Main}, we obtain the desired upper bound $\pi(\Delta_\lambda) \leq \pi(\mathcal{F}_k^{\le t}) = k!/k^k$.
\end{proof}

\subsection{Proof of Lemma~\ref{LEMMA:Final-Inequality}}
\noindent It remains to prove Lemma~\ref{LEMMA:Final-Inequality}.
We start with the following technical lemma.

\begin{lemma}\label{LEMMA:Algorithm}
Let $k \geq 2$, and let $t = t(k)$ be defined as in~\eqref{equ:t(k)}.
Consider the multiset
$$
T = \left\{ 1^{(k!)},\ 2^{(k! / 2)},\ \ldots,\ k^{(k! / k)} \right\},
$$
where each $i \in [k]$ appears with multiplicity $k! / i$.
Then there exist $k!$ multisets $A_1, \ldots, A_{k!}$ satisfying
    \begin{itemize}
        \item[(1).] $\bigcup_{i=1}^{k!} A_i = T$ and $A_i\cap A_j =\emptyset$ for all $1\le i<j\le k!$;
        \item[(2).] For each $i \in [k!]$, the summation (counting with multiplicities) $\omega(A_i) \coloneqq \sum_{x \in A_i} x$ equals $k$;
        \item[(3).] For each $i \in [k!]$, $A_i$ contains at most one element greater than $t$.
    \end{itemize}
\end{lemma}

\begin{proof}
    Define $m\coloneqq \left( \sum_{i=t+1}^k \frac{1}{i} \right) \cdot k!$ and $N\coloneqq \left( \sum_{i=1}^k \frac{1}{i} \right) \cdot k!$.
    So $m\leq k!$, and the multiset $T$ contains $N$ elements, among which exactly $m$ elements are strictly greater than $t$.
    We order the elements of $T = \{x_1, \ldots, x_N\}$ such that $k = x_1 \geq x_2 \geq \cdots \geq x_m = t+1 > t=x_{m+1}  \geq \cdots \geq x_N = 1.$

    To obtain the desired multisets $A_1, \ldots, A_{k!}$, we apply induction to construct a nested sequence $A_i^{(0)} \subseteq A_i^{(1)} \subseteq \cdots \subseteq A_i^{(N-m)}=A_i$ for each $i \in [k!]$, satisfying the following for all $0 \le s \le N-m$:
    \begin{itemize}
        \item $\bigcup_{i=1}^{k!} A_i^{(s)} = \{x_1, x_2, \ldots, x_{m+s}\}$ and $A_i^{(s)} \cap A_j^{(s)} = \emptyset$ for all $1 \le i < j \le k!$;

        \item For each $i \in [k!]$, the summation $\omega(A_i^{(s)}) \le k$;

        \item Each $A_i^{(s)}$ contains at most one element greater than $t$.
    \end{itemize}

    We initialize the construction by defining
    $A_i^{(0)} = \begin{cases}  \{x_i\} & \text{for } 1 \leq i \leq m, \\ \emptyset & \text{for } m < i \leq k!
    \end{cases} $.
    
To proceed, we assume that for some $0 \leq s \leq N-m-1$, the collection $\{A_i^{(s)}\}_{i=1}^{k!}$ has been constructed to satisfy the inductive hypotheses.

Let the value of $x_{m+s+1}$ be $z\geq 1$. We claim that  there exists an index $p \in [k!]$ such that $\omega(A_p^{(s)}) \leq k - z$.
Suppose, for contradiction, that $\omega(A_i^{(s)}) \geq k + 1 - z$ for all $i \in [k!]$. 
By hypothesis, the sets $A_i^{(s)}$ form a partition of $\{x_1, \ldots, x_{m+s}\}$, and $T \setminus \{x_1, \ldots, x_{m+s},x_{m+s+1}\}$ contains exactly $\frac{k!}{j}$ copies of the element $j$ for every $1 \leq j \leq z-1$.
So we have
\begin{align*}
        k \cdot k! = \omega(T) = \sum_{i=1}^{k!} \omega(A_i^{(s)}) + \sum_{i=m+s+1}^{N} x_i
        \geq k! \cdot (k + 1 - z) + z + \sum_{j=1}^{z-1} \frac{k!}{j} \cdot j = k \cdot k! + z> k \cdot k!,
\end{align*}
which yields the desired contradiction. This proves the claim.

We can then define $
        A_i^{(s+1)} = \begin{cases}
        A_p^{(s)} \cup \{x_{m+s+1}\} & \text{if } i = p, \\
        A_i^{(s)} & \text{otherwise}
        \end{cases}
    $.
One can readily verify that these sets satisfy all inductive conditions.
Finally, setting $A_i = A_i^{(N-m)}$ for all $i \in [k!]$. 
We observe that $\omega(A_i) \leq k$ for each $i\in [k!]$ and $\sum_{i=1}^{k!} \omega(A_i) = \omega(T) = k \cdot k!$.
Consequently, it implies $\omega(A_i) = k$ for all $i \in [k!]$, thus proving Lemma~\ref{LEMMA:Algorithm}.
\end{proof}

Now we are ready to prove Lemma~\ref{LEMMA:Final-Inequality}.

\begin{proof}[\bf Proof of Lemma~\ref{LEMMA:Final-Inequality}.]
Let $T$ and $A_1,A_2,\ldots,A_{k!}$ be the multisets given by Lemma~\ref{LEMMA:Algorithm}. 
For any indices $1 \leq i_1,i_2,\ldots,i_r \leq t$ and $j \in [k]$ satisfying $i_1+i_2+\cdots+i_r+j \leq k$, 
using the lemma condition, we have the chain of inequalities
$$
y_{i_1}+y_{i_2}+\cdots+y_{i_r}+y_{j}
\leq y_{i_1}+y_{i_2}+\cdots+y_{i_{r-1}}+y_{i_r+j}
\leq \cdots \leq y_{i_1+i_2+\cdots+i_r+j}.
$$
Since each $A_i$ contains at most one element greater than $t$, we obtain the bound
$
\sum_{j\in A_i} y_j \leq y_{\omega(A_i)} = y_k $ for each $ i \in [k!].
$
Summing over all $i \in [k!]$, we further derive that
\[
\sum_{i=1}^k \frac{y_i}{i}
=\frac{1}{k!}\sum_{i=1}^k \frac{k!}{i}y_i 
=\frac{1}{k!}\sum_{j\in T} y_j
=\frac{1}{k!}\sum_{i=1}^{k!} \left(\sum_{j\in A_i} y_j\right) \leq y_k.
\]
Applying the AM-GM inequality to this gives
$\prod_{i=1}^k y_i = k!\prod_{i=1}^k \frac{y_i}{i} \le k!\left(\frac{\sum_{i=1}^k \frac{y_i}{i}}{k}\right)^k\le \frac{k!}{k^k}y_k^k.$
\end{proof}

We remark that the constant $t(k)$ in Lemma~\ref{LEMMA:Final-Inequality} cannot be reduced, even by one. 
The construction is essentially identical to that of Lemma 3.10 in \cite{ilkovivc2025improved}, so we omit the detailed proof.
This establishes that the specific value $t(k)$ in Theorem~\ref{THM:Main} is indeed optimal for the underlying entropy method.

\bigskip

{\noindent\bf Acknowledgments.}
We are grateful to Peter Frankl and Xizhi Liu for helpful discussions.

\end{document}